\documentclass[10pt]{amsart}
\usepackage{latexsym}
\usepackage{amsfonts}
\usepackage{amssymb}
\usepackage{amsmath}
\usepackage{mathrsfs}
\usepackage{hyperref}

\newtheorem{definition}{Definition}[section] 
\newtheorem{theorem}[definition]{Theorem}
\newtheorem{lemma}[definition]{Lemma}
\newtheorem{proposition}[definition]{Proposition}

\newtheorem{remark}[definition]{Remark}

\title[On some supersymmetric spectral data : A multiplicativity property]{On the supersymmetric $N=1$ and $N=(1,1)$ spectral data :\\
A multiplicativity property}
\author{Satyajit Guin}
\address{Indian Institute of Science Education and Research, Mohali, Punjab 140306}
\email{satyamath@gmail.com , satyajit@iisermohali.ac.in}
\thanks{$\dagger$ Author is supported by INSPIRE Faculty Award (Award No. DST/INSPIRE/04/2015/000901)}

\date{\today}

\begin{document}
\rmfamily

\begin{abstract}
We show that there are six different choice of tensor product of supersymmetric $N=(1,1)$ spectral data in the context of supersymmetric quantum
theory and noncommutative geometry. We also show that the procedure of extending a supersymmetric $N=1$ spectral data to $N=(1,1)$ spectral data
respects only one tensor product among these. We refer this as the multiplicativity property of the extension procedure. Therefore, if we demand
that the extension procedure is multiplicative then there is a unique choice of tensor product of $N=(1,1)$ spectral data.
\end{abstract}

\maketitle

\section{Introduction}

In noncommutative geometry, a (noncommutative) manifold is described by a tuple called spectral triple. It turns out that the notion of spectral
triple is not quite appropriate to describe higher geometric structures, e.g. symplectic, complex, Hermitian, K\"ahler or hyper-K\"ahler, even in the
classical setting. Inspired by the work of Witten (\cite{Wit1}) and Jaffe et al. (\cite{Jaf}), a natural solution has been obtained by Fr\"{o}hlich
et al. (\cite{FGR1},\cite{FGR2}) in the context of supersymmetric quantum mechanics. Supersymmetric algebraic formulation of manifolds endowed with
these higher geometric structures is obtained in (\cite{FGR1}), which then readily generalizes to the noncommutative geometry framework in (\cite{FGR2}).
These various higher geometric structures are denoted by $N=1,\,N=2$ and $N=(n,n)$ with $n=1,2,4$, along the line of supersymmetry.
Note that the $N=1$ spectral data is specified by a $\varTheta$-summable even spectral triple in noncommutative geometry.

The $N=(1,1)$ spectral data is the first step of defining higher geometric structure on a $N=1$ spectral data (i,e. spectral triple). In the
classical case of spin manifold $\mathbb{M}$ from the $N=(1,1)$ spectral data one may recover the graded algebra of differential forms on
$\mathbb{M}$ and in particular the exterior differential. Hence, it is a natural and important question whether a $N=1$ spectral data extends to
$N=(1,1)$ spectral data over the same (noncommutative) base space. As shown in (\cite{FGR1}), this is always possible in the classical case of
manifolds. However, in the noncommutative situation one faces a difficulty regarding the extension. Guided by the classical case, a procedure to
extend a $N=1$ spectral data to a $N=(1,1)$ spectral data over the same base space has been suggested in (\cite{FGR2}) using suitable connection
on a dense finitely generated projective module equipped with a Hermitian structure. Apart from the classical case of manifolds, existence of such
connection was proved for the noncommutative $2$-torus and the fuzzy $3$-sphere in (\cite{FGR2}).

Now, like in the classical case where forming the product between two geometric spaces is a basic operation in geometry, considering product of
noncommutative spaces is also of much relevant importance not only for construction of a would-be tensor category but also bears interest for some
applications in theoretical physics (\cite{CC},\cite{CMA},\cite{Con1}). We study the behavior of the above discussed extension procedure under tensor
product of $N=1$ spectral data. We use the shorthand notation $\Phi:N=1\Longrightarrow N=(1,1)$ to mean this extension procedure. For two $N=1$
spectral data $(\mathcal{A}_j,\mathcal{H}_j,D_j,\gamma_j),\,j=1,2$, we show that $\Phi\left(\otimes_{j=1}^2(\mathcal{A}_j,\mathcal{H}_j,D_j,
\gamma_j)\right)$ also becomes a $N=(1,1)$ spectral data if $\Phi(\mathcal{A}_j,\mathcal{H}_j,D_j,\gamma_j)$ are individually so. We also show that
there are six different choice of tensor product of $N=(1,1)$ spectral data and $\Phi$ becomes multiplicative, i,e.
\begin{center}
$\Phi\left(\otimes_{j=1}^2(\mathcal{A}_j,\mathcal{H}_j,D_j,\gamma_j)\right)=\otimes_{j=1}^2\Phi\left((\mathcal{A}_j,\mathcal{H}_j,D_j,\gamma_j)\right)$
\end{center}
w.r.t only one choice of tensor product among these. That is, if we demand that the procedure of extention $\Phi$ is multiplicative then there is
a unique choice of tensor product of $N=(1,1)$ spectral data.
\medskip


\section {The supersymmetric \texorpdfstring{$N=1$}{Lg} and \texorpdfstring{$N=(1,1)$}{Lg} spectral data}

\begin{definition}\label{N=1 spectral data}
A quadruple $(\mathcal{A},\mathcal{H},D,\gamma)$ is called a set of $N=1$ spectral data if
\begin{enumerate}
\item $\mathcal{A}$ is a unital associative $*$-algebra represented faithfully on the separable Hilbert space $\mathcal{H}$ by bounded operators;
\item $D$ is a self-adjoint operator on $\mathcal{H}$ such that
\begin{enumerate}
\item for each $a\in\mathcal{A}$, the commutator $[D,a]$ extends uniquely to a bounded operator on $\mathcal{H}$,
\item the operator $exp(-\varepsilon D^2)$ is trace class for all $\varepsilon>0;$
\end{enumerate}
\item $\gamma$ is a $\mathbb{Z}_2$-grading on $\mathcal{H}$ such that $[\gamma,a]=0$ for all $a\in\mathcal{A}$ and $\{\gamma,D\}=0$.
\end{enumerate}
\end{definition}

\begin{remark}\rm
Observe that the $N=1$ spectral data represents a $\varTheta$-summable even spectral triple in noncommutative geometry.
\end{remark}

\begin{definition}\label{N=(1,1) spectral data}
A quintuple $(\mathcal{A},\mathcal{H},d,\gamma,\star)$ is called a set of $N=(1,1)$ spectral data if
\begin{enumerate}
\item $\mathcal{A}$ is a unital associative $*$-algebra represented faithfully on the separable Hilbert space $\mathcal{H}$ by bounded operators;
\item $d$ is a densely defined closed operator on $\mathcal{H}$ such that
\begin{enumerate}
\item $d^2=0$,
\item for each $a\in\mathcal{A}$, the commutator $[d,a]$ extends uniquely to a bounded operator on $\mathcal{H}$,
\item the operator $exp(-\varepsilon\bigtriangleup)$, with $\bigtriangleup=dd^*+d^*d$, is trace class for all $\varepsilon>0;$
\end{enumerate}
\item $\gamma$ is a $\mathbb{Z}_2$-grading on $\mathcal{H}$ such that $[\gamma,a]=0$ for all $a\in\mathcal{A}$ and $\{\gamma,d\}=0;$
\item $\star$ is a unitary operator acting on $\mathcal{H}$ such that $[\star,a]=0$ for all $a\in\mathcal{A}$ and $\,\star\,d=-d^*\star$.
\end{enumerate}
\end{definition}

\begin{remark}\rm
\begin{enumerate}
\item In analogy with the classical case, the operator $\star$ is called the Hodge operator.
\item The operator $\star$ can never be taken as the grading $\gamma$ itself. Otherwise, the condition $\{\gamma,d\}=0$ and $\star d=-d^*\star\,$
will force $d=d^*$, in which case $d=0$ given $d^2=0$.
\end{enumerate}
\end{remark}

As is always achievable in the classical case of manifolds, we can, and we will, w.l.o.g assume that the Hodge operator is a self-adjoint unitary
commuting with the grading operator (see discussion in page $(139)$ of \cite{FGR2}).

Definition (\ref{N=(1,1) spectral data}) of $N=(1,1)$ spectral data has an alternative description. One can introduce two unbounded operators
\begin{center}
$\mathfrak{D}=d+d^*\quad,\quad\overline{\mathfrak{D}}=i(d-d^*)$
\end{center}
(Caution: $\overline{\mathfrak{D}}$ is not the closure of $\mathfrak{D}$) which satisfy the following relations
\begin{center}
$\mathfrak{D}^2=\overline{\mathfrak{D}}^2\quad,\quad\{\mathfrak{D},\overline{\mathfrak{D}}\}=0$
\end{center}
making the notion of $N=(1,1)$ spectral data an immediate generalization of a classical $N=(1,1)$ Dirac bundle (\cite{FGR1},\cite{FGR2}). Conversely,
starting with $\mathfrak{D},\,\overline{\mathfrak{D}}$ satisfying the above relations, one can define
\begin{center}
$d=\frac{1}{2}(\mathfrak{D}-i\overline{\mathfrak{D}}\,)\quad,\quad d^*=\frac{1}{2}(\mathfrak{D}+i\overline{\mathfrak{D}}\,)\,.$
\end{center}
For all $\varepsilon>0$, the condition $exp(-\varepsilon(dd^*+d^*d))$ is a trace class operator becomes equivalent with $exp(-\varepsilon
\mathfrak{D}^2)$ is a trace class operator.

\begin{lemma}\label{crucial relation involving Hodge}
We have
\begin{enumerate}
\item $\{\gamma,d\}=0$ if and only if $\{\gamma,\mathfrak{D}\}=\{\gamma,\overline{\mathfrak{D}}\}=0\,;$
\item $\star d=-d^*\star\,$ if and only if $\{\star,\mathfrak{D}\}=[\star,\overline{\mathfrak{D}}\,]=0\,$.
\end{enumerate}
\end{lemma}
\begin{proof}
Easy verification.
\end{proof}

Therefore, the data $(\mathcal{A},\mathcal{H},d,\gamma,\star)$ and $(\mathcal{A},\mathcal{H},\mathfrak{D},\overline{\mathfrak{D}},\gamma,\star)$ are
equivalent. In the classical situation of manifolds, any $N=1$ spectral data can always be extended to a $N=(1,1)$ spectral data over the same
base space (\cite{FGR1}). However, in the noncommutative framework this extension is not obvious. Guided by the classical case of manifolds, a
procedure of extension is suggested by Fr\"{o}hlich et al. in (\cite{FGR2}) which we discuss now.
\medskip

Let $\mathcal{E}$ be a finitely generated projective (f.g.p) left module over $\mathcal{A}$ and $\mathcal{E}^*:=\mathcal{H}om_\mathcal{A}
(\mathcal{E},\mathcal{A})$. Clearly, $\mathcal{E}^*$ is also a left $\mathcal{A}$-module by the rule $(a\,.\,\phi)(\xi):=\phi(\xi)a^*,\,\forall
\,\xi\in\mathcal{E}$. Recall the definition of Hermitian structure (Def. $[2.8]$ in \cite{FGR2}) on $\mathcal{E}$. Any free $\mathcal{A}$-module
$\mathcal{E}_0=\mathcal{A}^n$ has a canonical Hermitian structure on it, given by $\langle\,\xi,\eta\,\rangle_\mathcal{A}=\sum_{j=1}^n\xi_j\eta_j^*$
for all $\xi=(\xi_1,\ldots,\xi_q)\in\mathcal{E}_0\,,\,\eta=(\eta_1,\ldots,\eta_q)\in\mathcal{E}_0$. By definition, any f.g.p module $\mathcal{E}$
can be written as $\mathcal{E}=p\mathcal{A}^n$ for some idempotent $p\in M_n(\mathcal{A})$. If this idempotent $p$ is a projection, i,e. $p=p^2=p^*$
then one can restrict the canonical Hermitian structure on $\mathcal{A}^n$ to $\mathcal{E}$, and $\mathcal{E}$ becomes a Hermitian f.g.p module.
Under the hypothesis that $\mathcal{A}$ is stable under holomorphic functional calculus in a $C^*$-algebra $A$, we have the following existence
lemma of Hermitian structure (Lemma $2.2(b)$ in \cite{CG}).

\begin{lemma}$($\cite{CG}$)$\label{existence lemma}
Every f.g.p module $\mathcal{E}$ over $\mathcal{A}$ is isomorphic as a f.g.p module with $p\mathcal{A}^n$, where $p\in M_n(\mathcal{A})$ is a
self-adjoint idempotent i,e. a projection. Hence, $\mathcal{E}$ has a Hermitian structure on it.
\end{lemma}

\begin{remark}\rm
Without the condition of stability under the holomorphic functional calculus, existence of Hermitian structure on an arbitrary f.g.p module is not
guaranteed.
\end{remark}

Recall the following important structure theorem of Hermitian f.g.p module (Th. $[3.3]$ in \cite{CG}).

\begin{theorem}$($\cite{CG}$)$\label{structure thm of herm. mod.}
Let $\,\mathcal{E}$ be a f.g.p $\mathcal{A}$-module with a Hermitian structure on it and $\mathcal{A}$ is stable under the holomorphic functional
calculus in a $C^*$-algebra $A$. Then we have a self-adjoint idempotent $p\in M_n(\mathcal{A})$ such that $\mathcal{E}\cong p\mathcal{A}^n$
as f.g.p module, and $\mathcal{E}$ has the induced canonical Hermitian structure.
\end{theorem}

In his book (\cite{Con1}), Connes has suggested that in the context of Hermitian f.g.p module one should always work with spectrally invariant
algebras, i,e. subalgebras of $C^*$-algebras stable under the holomorphic functional calculus. The reason is that all possible notions of positivity
will coincide in that case. Moreover, we will also have Th. (\ref{structure thm of herm. mod.}) which makes computations involving the Hermitian
structure much easier. Incorporating Connes' suggestion we will always work with spectrally invariant algebras in this article. Note that in the
classical case of manifold $\mathbb{M},\,C^\infty(\mathbb{M})$ is spectrally invariant subalgebra of the unital $C^*$-algebra $C(\mathbb{M})$.
Let $\Omega_{D}^1(\mathcal{A})$ be the $\mathcal{A}$-bimodule $\{\sum\,a_j[D,b_j]:a_j,b_j\in\mathcal{A}\}$ of noncommutative
$1$-forms and $d:\mathcal{A}\rightarrow\Omega_{D}^1(\mathcal{A})$, given by $a\mapsto[D,a]$, be the Dirac dga differential (\cite{Con1}).
Note that $(da)^*=-da^*$ by convention.

\begin{definition}
Let $\,\mathcal{E}$ be a f.g.p left module over $\mathcal{A}$ with a Hermitian structure $\langle\,\,\,,\,\,\rangle_\mathcal{A}$ on it. A compatible
connection on $\mathcal{E}$ is a $\mathbb{C}$-linear map $\nabla:\mathcal{E}\longrightarrow\Omega_{D}^1(\mathcal{A})\otimes_\mathcal{A}
\mathcal{E}$ satisfying
\begin{enumerate}
\item[(a)] $\nabla(a\xi)=a(\nabla\xi)+da\otimes\xi\,,\,\,\forall\,\xi\in\mathcal{E},\,a\in\mathcal{A};$
\item[(b)] $\langle\,\nabla\xi,\eta\,\rangle-\langle\,\xi,\nabla\eta\,\rangle=d\langle\,\xi,\eta\,\rangle_\mathcal{A}\,\,\,\forall\,\xi,\eta\in
\mathcal{E}$.
\end{enumerate}
\end{definition}

Meaning of equality $(b)$ in $\Omega_D^1(\mathcal{A})$ is, if $\nabla(\eta)=\sum\omega_j\otimes\eta_j\in\Omega_D^1(\mathcal{A})\otimes
\mathcal{E}$, then $\langle\,\xi,\nabla\eta\,\rangle=\sum\,\langle\xi,\eta_j\rangle_\mathcal{A}\,\omega_j^*$. Compatible connection always exists
(\cite{Con1}). The space of all compatible connections on $\mathcal{E}$, which we denote by $C(\mathcal{E})$, is an affine space with associated
vector space $\mathcal{H}om_\mathcal{A}(\mathcal{E},\Omega_D^1(\mathcal{A})\otimes_\mathcal{A}\mathcal{E})$.
\medskip

\textbf{A procedure to extend a $N=1$ spectral data to $N=(1,1)$ spectral data~:}\\
Start with a $N=1$ spectral data $(\mathcal{A},\mathcal{H},D,\gamma)$ equipped with a real structure $J$ (\cite{Con2},\cite{Con3}). That is,
there exists an anti-unitary operator $J$ on $\mathcal{H}$ such that
\begin{center}
$J^2=\varepsilon I\quad,\quad JD=\varepsilon^\prime DJ\quad,\quad J\gamma=\varepsilon^{\prime\prime}\gamma J\quad$
\end{center}
for some signs $\,\varepsilon,\varepsilon^\prime,\varepsilon^{\prime\prime}=\pm 1$ depending on $KO$-dimension $n\in\mathbb{Z}_8\,:$
\begin{table}[ht]
\centering
\begin{tabular}{c| c c c c c c c c c}
\hline
$n$ & $0$ & $2$ & $4$ & $6$ &$\,$ & $1$ & $3$ & $5$ & $7$ \\ [0.8ex]
\hline
$\varepsilon$ & $+$ & $-$ & $-$ & $+$ & $\,$ & $+$ & $-$ & $-$ & $+$\\
$\varepsilon^\prime$  & $+$ & $+$ & $+$ & $+$ & $\,$ & $-$ & $+$ & $-$ & $+$\\
$\varepsilon^{\prime\prime}$ & $+$ & $-$ & $+$ & $-$ & $\,$ & $\,$ & $\,$ & $\,$ & $\,$\\
\hline
\end{tabular}
\end{table}\\
and satisfying $[JaJ^*,b]=[JaJ^*,[D,b]]=0\,\,\forall\,a,b\in\mathcal{A}$. The real structure $J$ now enables us to equip the Hilbert space
$\mathcal{H}$ with an $\mathcal{A}$-bimodule structure
\begin{center}
$a\,.\,\xi\,.\,b:=\pi(a)Jb^*J^*\xi\,.$
\end{center}
We can extend this to a right action of $\,\Omega_D^1(\mathcal{A}):=\{\sum_j\,a_j[D,b_j]:a_j,b_j\in\mathcal{A}\}$ on $\mathcal{H}$ by the rule
\begin{center}
$\xi\,.\,\omega:=J\omega^*J^*\xi\,.$
\end{center}
Assume that $\mathcal{H}$ contains a dense f.g.p left $\mathcal{A}$-module $\mathcal{E}$ which is stable under $J$ and $\gamma$. In particular,
$\mathcal{E}$ is itself an $\mathcal{A}$-bimodule. Since, $\mathcal{A}$ is spectrally invariant subalgebra in a $C^*$-algebra $A$ we have
a Hermitian structure $\langle\,\,,\,\rangle_\mathcal{A}$ on $\mathcal{E}$ (by Lemma [\ref{existence lemma}]), which makes $\mathcal{E}
\otimes_\mathcal{A}\mathcal{E}$ into an inner-product space by the following rule~:
\begin{eqnarray}\label{the inner product}
\langle\xi\otimes\eta\,,\,\xi^\prime\otimes\eta^\prime\rangle:=\langle\eta\,,\,\langle\xi,\xi^\prime\rangle_\mathcal{A}(\eta^\prime)\rangle\,.
\end{eqnarray}
Let $\,\widetilde{\mathcal{H}}:=\overline{\mathcal{E}
\otimes_\mathcal{A}\mathcal{E}}^{\langle\,\,,\,\rangle}$. Define the following anti-linear flip operator
\begin{align*}
\Psi:\Omega_D^1(\mathcal{A})\otimes_\mathcal{A}\mathcal{E} &\longrightarrow \mathcal{E}\otimes_\mathcal{A}\Omega_D^1(\mathcal{A})\\
\omega\otimes\xi &\longmapsto J\xi\otimes\omega^*
\end{align*}
It is easy to verify that $\Psi$ is well-defined and satisfies $\Psi(as)=\Psi(s)a^*,\,\,\forall\,s\in\Omega_D^1(\mathcal{A})\otimes_\mathcal{A}
\mathcal{E}$. Consider a compatible connection
\begin{center}
$\nabla:\mathcal{E}\longrightarrow\Omega_D^1(\mathcal{A})\otimes_\mathcal{A}\mathcal{E}$
\end{center}
such that $\nabla$ commutes with the grading $\gamma$ on $\mathcal{E}\subseteq\mathcal{H}$, i,e. $\nabla\gamma\xi=(1\otimes\gamma)\nabla\xi,\,
\forall\,\xi\in\mathcal{E}$. For each such connection $\nabla$ on $\mathcal{E}$, there is the following associated right-connection
\begin{align*}
\overline{\nabla}:\mathcal{E} &\longrightarrow \mathcal{E}\otimes_\mathcal{A}\Omega_D^1(\mathcal{A})\\
\xi &\longmapsto -\Psi(\nabla J^*\xi)
\end{align*}
Thus, we get a $\mathbb{C}$-linear map (the so called ``tensored connection'')
\begin{align*}
\widetilde\nabla:\mathcal{E}\otimes_\mathcal{A}\mathcal{E} &\longrightarrow \mathcal{E}\otimes_\mathcal{A}\Omega_D^1(\mathcal{A})\otimes_\mathcal{A}
\mathcal{E}\\
\xi_1\otimes\xi_2 &\longmapsto \overline{\nabla}\xi_1\otimes\xi_2+\xi_1\otimes\nabla\xi_2
\end{align*}
Note that $\widetilde\nabla$ is not a connection in the usual sense because of the position of $\,\Omega_D^1(\mathcal{A})$. Define the following two
$\mathbb{C}$-linear maps
\begin{align*}
c\,,\,\overline{c}\,:\mathcal{E}\otimes_\mathcal{A}\Omega_D^1(\mathcal{A})\otimes_\mathcal{A}\mathcal{E} &\longrightarrow \mathcal{E}
\otimes_\mathcal{A}\mathcal{E}\\
c:\xi_1\otimes\omega\otimes\xi_2 &\longmapsto \xi_1\otimes\omega\,.\,\xi_2\\
\overline{c}:\xi_1\otimes\omega\otimes\xi_2 &\longmapsto \xi_1\,.\,\omega\otimes\gamma\xi_2
\end{align*}
Now, introduce the following densely defined unbounded operators on $\widetilde{\mathcal{H}}$
\begin{center}
$\mathfrak{D}:=c\circ\widetilde{\nabla}\quad,\quad\overline{\mathfrak{D}}:=\overline{c}\circ\widetilde{\nabla}$
\end{center}
(Caution: $\overline{\mathfrak{D}}$ is not the closure of $\mathfrak{D}$). In order to obtain a set of $N=(1,1)$ spectral data on $\mathcal{A}$,
one has to find a specific connection $\nabla$ on a suitable dense f.g.p left $\mathcal{A}$-module $\mathcal{E}$ such that
\begin{itemize}
\item[(a)] The operators $\mathfrak{D}$ and $\overline{\mathfrak{D}}$ become essentially self-adjoint on $\widetilde{\mathcal{H}}$,
\item[(b)] The relations $\mathfrak{D}^2=\overline{\mathfrak{D}}^2$ and $\{\mathfrak{D},\overline{\mathfrak{D}}\}=0$ are satisfied.
\end{itemize}
The $\mathbb{Z}_2$-grading on $\widetilde{\mathcal{H}}$ is simply the tensor product grading $\,\widetilde{\gamma}:=\gamma\otimes\gamma$, and the
Hodge operator is taken to be $\,\star:=1\otimes\gamma$ (In \cite{FGR2}, this is mistakenly taken as $\star=\gamma\otimes 1$). The sextuple
$(\mathcal{A},\widetilde{\mathcal{H}},\mathfrak{D},\overline{\mathfrak{D}},\widetilde{\gamma},\star)$ is a candidate of a $N=(1,1)$ spectral data
extending the $N=1$ spectral data $(\mathcal{A},\mathcal{H},D,\gamma)$. The Hodge operator $\star$ additionally satisfies $\,\star^2=1$ and
$[\star,\gamma]=0$.
\medskip

We denote this procedure to extend a $N=1$ spectral data to $N=(1,1)$ spectral data over the same base space $\mathcal{A}$ by the shorthand
notation $\Phi:N=1\Longrightarrow N=(1,1)$.

\begin{definition}\label{product of N=1 data}$($\cite{Con1}$)$
For two $N=1$ spectral data $(\mathcal{A}_j,\mathcal{H}_j,D_j,\gamma_j),\,j=1,2,$ their Kasparov product is defined as $(\mathcal{A}_1\otimes
\mathcal{A}_2\,,\,\mathcal{H}_1\otimes\mathcal{H}_2\,,\,D_1\otimes 1+\gamma_1\otimes D_2\,,\,\gamma_1\otimes\gamma_2)$.
\end{definition}

Note that one can also take the Dirac operator $D_1\otimes\gamma_2+1\otimes D_2\,$. In that case the product $N=1$ spectral data $(\mathcal{A}_1\otimes
\mathcal{A}_2\,,\,\mathcal{H}_1\otimes\mathcal{H}_2\,,\,D_1\otimes 1+\gamma_1\otimes D_2\,,\,\gamma_1\otimes\gamma_2)$ and $(\mathcal{A}_1\otimes
\mathcal{A}_2\,,\,\mathcal{H}_1\otimes\mathcal{H}_2\,,\,D_1\otimes\gamma_2+1\otimes D_2\,,\,\gamma_1\otimes\gamma_2)$ become unitary equivalent, and
one such unitary acting on $\mathcal{H}_1\otimes\mathcal{H}_2$ is given by (\cite{Van})
\begin{center}
$U:=\frac{1}{2}(1\otimes 1+\gamma_1\otimes 1+1\otimes\gamma_2-\gamma_1\otimes\gamma_2)\,.$
\end{center}

\begin{definition}\label{product of N=(1,1) data}
For two $N=(1,1)$ spectral data $(\mathcal{A}_j,\mathcal{H}_j,\mathfrak{D}_j,\overline{\mathfrak{D}}_j,\gamma_j,\star_j),\,j=1,2,$ we define their
product to be $(\mathcal{A}:=\mathcal{A}_1\otimes\mathcal{A}_2\,,\mathcal{H}:=\mathcal{H}_1\otimes\mathcal{H}_2\,,\,\mathfrak{D}:=\mathfrak{D}_1
\otimes 1+\star_1\otimes\mathfrak{D}_2\,,\,\overline{\mathfrak{D}}:=\overline{\mathfrak{D}_1}\otimes\star_2+\gamma_1\otimes\overline{\mathfrak{D}_2}
\,,\,\gamma:=\gamma_1\otimes\gamma_2\,,\,\star:=\star_1\otimes\star_2)$.
\end{definition}

\begin{lemma}
The above defined product is well-defined.
\end{lemma}
\begin{proof}
Since, $\{\star_1,\mathfrak{D}_1\}=[\,\star_2,\overline{\mathfrak{D}_2}\,]=0$ we have $\,\mathfrak{D}^2=\overline{\mathfrak{D}}^{\,2}$. Now,
$[\,\star_1,\overline{\mathfrak{D}_1}\,]=\{\star_2,\mathfrak{D}_2\}=0$ implies that $\,\{\mathfrak{D},\overline{\mathfrak{D}}\}=0$. Finally,
\begin{center}
$Tr\left(exp(-\varepsilon\mathfrak{D}^2)\right)=Tr\left(exp(-\varepsilon\mathfrak{D}_1^2)\right)Tr\left(exp(-\varepsilon\mathfrak{D}_2^2)\right)<\infty$
\end{center}
for all $\varepsilon>0,\,\{\gamma,\mathfrak{D}\}=\{\gamma,\overline{\mathfrak{D}}\}=0$ and $\{\star,\mathfrak{D}\}=[\star,\overline{\mathfrak{D}}\,]=0$.
\end{proof}

\begin{lemma}
The associated $N=1$ spectral data of the product of two $N=(1,1)$ spectral data is unitary equivalent with the product of the associated $N=1$
spectral data of the individual $N=(1,1)$ spectral data.
\end{lemma}
\begin{proof}
Let $(\mathcal{A}_j,\mathcal{H}_j,\mathfrak{D}_j,\overline{\mathfrak{D}}_j,\gamma_j,\star_j),\,j=1,2,$ be two $N=(1,1)$ spectral data. The associated
$N=1$ spectral data are $(\mathcal{A}_j,\mathcal{H}_j,\mathfrak{D}_j,\gamma_j),\,j=1,2$. The Dirac operator of the product of these $N=1$ spectral
data is $\mathfrak{D}_1\otimes 1+\gamma_1\otimes\mathfrak{D}_2$ (Def. [\ref{product of N=1 data}]). Now, consider the product $\otimes_{j=1}^2
(\mathcal{A}_j,\mathcal{H}_j,\mathfrak{D}_j,\overline{\mathfrak{D}}_j,\gamma_j,\star_j)$ by Definition (\ref{product of N=(1,1) data}). The associated
$N=1$ spectral data of this product is $(\mathcal{A}_1\otimes\mathcal{A}_2\,,\mathcal{H}_1\otimes\mathcal{H}_2\,,\,\mathfrak{D}:=\mathfrak{D}_1
\otimes 1+\star_1\otimes\mathfrak{D}_2\,,\gamma_1\otimes\gamma_2)$. The operators $\mathfrak{D}_1\otimes 1+\star_1\otimes\mathfrak{D}_2$ and
$\mathfrak{D}_1\otimes\gamma_2+1\otimes\mathfrak{D}_2$ are unitary equivalent by the unitary
\begin{center}
$U:=\frac{1}{2}(1\otimes 1+\star_1\otimes 1+1\otimes\gamma_2-\star_1\otimes\gamma_2)$
\end{center}
where as, the operators $\mathfrak{D}_1\otimes\gamma_2+1\otimes\mathfrak{D}_2$ and $\mathfrak{D}_1\otimes 1+\gamma_1\otimes\mathfrak{D}_2$ are
unitary equivalent by the unitary
\begin{center}
$V:=\frac{1}{2}(1\otimes 1+\gamma_1\otimes 1+1\otimes\gamma_2-\gamma_1\otimes\gamma_2)\,.$
\end{center}
Hence, $\mathfrak{D}_1\otimes 1+\star_1\otimes\mathfrak{D}_2$ and $\mathfrak{D}_1\otimes 1+\gamma_1\otimes\mathfrak{D}_2$ are also unitary equivalent
by the unitary $VU$.
\end{proof}

\begin{proposition}\label{various product of N=(1,1)}
There are various other choice of tensor product of $N=(1,1)$ spectral data w.r.t the natural choice of $\,\gamma:=\gamma_1\otimes\gamma_2$ and
$\,\star:=\star_1\otimes\star_2\,:$
\begin{enumerate}
\item $\mathfrak{D}:=\mathfrak{D}_1\otimes 1+\star_1\otimes\mathfrak{D}_2$ and $\,\overline{\mathfrak{D}}:=\overline{\mathfrak{D}_1}\otimes\gamma_2+
\star_1\otimes\overline{\mathfrak{D}_2}$
\item $\mathfrak{D}:=\mathfrak{D}_1\otimes\star_2+1\otimes\mathfrak{D}_2$ and $\,\overline{\mathfrak{D}}:=\overline{\mathfrak{D}_1}\otimes\gamma_2+
\star_1\otimes\overline{\mathfrak{D}_2}$
\item $\mathfrak{D}:=\mathfrak{D}_1\otimes\star_2+1\otimes\mathfrak{D}_2$ and $\,\overline{\mathfrak{D}}:=\overline{\mathfrak{D}_1}\otimes\star_2+
\gamma_1\otimes\overline{\mathfrak{D}_2}$
\item $\mathfrak{D}:=\mathfrak{D}_1\otimes 1+\gamma_1\otimes\mathfrak{D}_2$ and $\,\overline{\mathfrak{D}}:=\overline{\mathfrak{D}_1}\otimes 1+
\gamma_1\otimes\overline{\mathfrak{D}_2}$
\item $\mathfrak{D}:=\mathfrak{D}_1\otimes\gamma_2+1\otimes\mathfrak{D}_2$ and $\,\overline{\mathfrak{D}}:=\overline{\mathfrak{D}_1}\otimes\gamma_2+
1\otimes\overline{\mathfrak{D}_2}$
\end{enumerate}
\end{proposition}
\begin{proof}
Proof is straightforward using the relationship between $\gamma_j$ and $\star_j$ with $\mathfrak{D}_j$ and $\overline{\mathfrak{D}_j}\,$ for
$j=1,2$.
\end{proof}

Now, suppose we have two $N=1$ spectral data $(\mathcal{A}_j,\mathcal{H}_j,D_j,\gamma_j),\,j=1,2,$ such that $\Phi(\mathcal{A}_j,\mathcal{H}_j,D_j,
\gamma_j)$ indeed give us two $N=(1,1)$ spectral data. One can consider the product of $(\mathcal{A}_j,\mathcal{H}_j,D_j,\gamma_j),\,j=1,2,$ and ask
the following questions.

\begin{enumerate}
\item[]\textbf{Question $1$:} Does $\Phi\left(\otimes_{j=1}^2(\mathcal{A}_j,\mathcal{H}_j,D_j,\gamma_j)\right)$ also give a $N=(1,1)$ spectral data?
\item[]\textbf{Question $2$:} If answer to the previous question is affirmative, then is it always true that
\begin{center}
$\Phi\left(\otimes_{j=1}^2(\mathcal{A}_j,\mathcal{H}_j,D_j,\gamma_j)\right)=\otimes_{j=1}^2\Phi\left((\mathcal{A}_j,\mathcal{H}_j,D_j,\gamma_j)\right)$
\end{center}
i,e. the procedure to extend a $N=1$ spectral data to $N=(1,1)$ spectral data satisfies the multiplicativity property?
\end{enumerate}

In the next section we show that answer to Qn.$(1)$ is always affirmative, but answer to Qn.$(2)$ is affirmative w.r.t the Definition
(\ref{product of N=(1,1) data}), but not true w.r.t any other tensor product described in Proposition (\ref{various product of N=(1,1)}). Thus, if
we demand that the extension procedure $\Phi$ is multiplicative then there is a unique choice of tensor product of $N=(1,1)$ spectral data.
\medskip


\section {The Multiplicativity property}

Consider two $N=1$ spectral data $(\mathcal{A}_j,\mathcal{H}_j,D_j,\gamma_j),\,j=1,2$, equipped with real structures $J_j$ such that $\Phi(\mathcal{A}_j,
\mathcal{H}_j,D_j,\gamma_j)$ gives us two honest $N=(1,1)$ spectral data, obtained by the extension procedure. Now, consider two dense Hermitian
f.g.p modules $\mathcal{E}_j\subseteq\mathcal{H}_j$ over $\mathcal{A}_j$, stable under $J_j$ and $\gamma_j$ respectively. Then there are projections
$p_j\in M_{m_j}(\mathcal{A}_j)$ such that as f.g.p modules $\mathcal{E}_j=p_j\mathcal{A}_j^{m_j}$ and the Hermitian structures on them become the
induced canonical structure from the free modules $\mathcal{A}_j^{m_j}$ (Theorem [\ref{structure thm of herm. mod.}]). Clearly, $\mathcal{E}:=
\mathcal{E}_1\otimes\mathcal{E}_2$ is f.g.p module over $\mathcal{A}_1\otimes\mathcal{A}_2$ and has the canonical Hermitian structure induced by the
free module $(\mathcal{A}_1\otimes\mathcal{A}_2)^{m_1m_2}$. Moreover, this Hermitian structure has the following form
\begin{center}
$\langle\xi_1\otimes\eta_1\,,\,\xi_2\otimes\eta_2\rangle_{\mathcal{A}_1\otimes\mathcal{A}_2}=\langle\xi_1,\xi_2\rangle_{\mathcal{A}_1}\langle\eta_1,
\eta_2\rangle_{\mathcal{A}_2}\,,$
\end{center}
which can be easily verified. The real structure on the product of $N=1$ spectral data is given by $J_1\otimes J_2$ (\cite{Con2},\cite{Con3},\cite{Van},\cite{DD}).
Observe that $\mathcal{E}\subseteq\mathcal{H}_1\otimes\mathcal{H}_2$ is dense and stable under $J=J_1\otimes J_2$ and $\gamma=\gamma_1\otimes\gamma_2$.

\begin{lemma}\label{first form}
For the product of two $N=1$ spectral data $(\mathcal{A}_j,\mathcal{H}_j,D_j,\gamma_j),\,j=1,2$, the associated bimodule $\Omega_D^1(\mathcal{A})$
of noncommutative $1$-forms is isomorphic to $\Omega_{D_1}^1(\mathcal{A}_1)\otimes\mathcal{A}_2\bigoplus\mathcal{A}_1\otimes\Omega_{D_2}^1
(\mathcal{A}_2)$ as $\mathcal{A}=\mathcal{A}_1\otimes\mathcal{A}_2$-bimodule, where $D=D_1\otimes 1+\gamma_1\otimes D_2$.
\end{lemma}
\begin{proof}
Observe that $[D,\sum a_1\otimes a_2]=\sum [D_1,a_1]\otimes a_2+\gamma_1a_1\otimes[D_2,a_2]$. Since, $\gamma_1[D_1,a_1]=-[D_1,a_1]\gamma_1$ and
$\gamma_1^2=1$ it follows that
\begin{center}
$\Omega_D^1(\mathcal{A})\subseteq\Omega_{D_1}^1(\mathcal{A}_1)\otimes\mathcal{A}_2\bigoplus\mathcal{A}_1\otimes\Omega_{D_2}^1(\mathcal{A}_2)\,.$
\end{center}
In order to show the equality observe that any element $\sum a_0[D_1,a_1]\otimes a_2\in\Omega_{D_1}^1(\mathcal{A}_1)\otimes\mathcal{A}_2$ can be
written as $\sum (a_0\otimes a_2)[D,a_1\otimes 1]$ and similarly, $\sum a_0\otimes a_1[D_2,a_2]\in\mathcal{A}_1\otimes\Omega_{D_2}^1(\mathcal{A}_2)$
can be written as $\sum (a_0\otimes a_1)[D,1\otimes a_2]$. This proves the equality and one can check that the $\mathcal{A}$-bimodule structure is
preserved.
\end{proof}

\begin{lemma}\label{first differential}
The Dirac dga differential $\,d:\mathcal{A}\longrightarrow\Omega_D^1(\mathcal{A})$ is given by
\begin{center}
$d(a_1\otimes a_2)=(d_1(a_1)\otimes a_2\,,\,a_1\otimes d_2(a_2))$
\end{center}
where $d_j:\mathcal{A}_j\rightarrow\Omega_{D_j}^1(\mathcal{A}_j)$, for $j=1,2$, are the Dirac dga differentials associated with $\mathcal{A}_j$.
\end{lemma}
\begin{proof}
Follows from the previous Lemma (\ref{first form}). 
\end{proof}

\begin{lemma}\label{lemma 3}
For $\mathcal{A}=\mathcal{A}_1\otimes\mathcal{A}_2$, we have $\mathcal{E}\otimes_\mathcal{A}\mathcal{E}\cong(\mathcal{E}_1\otimes_{\mathcal{A}_1}
\mathcal{E}_1)\otimes_\mathbb{C}(\mathcal{E}_2\otimes_{\mathcal{A}_2}\mathcal{E}_2)$ as $\mathcal{A}$-bimodule.
\end{lemma}
\begin{proof}
Since $\mathcal{A}_1,\mathcal{A}_2$ both are unital algebras, there is a canonical isomorphism of $\mathcal{A}$-bimodule.
\end{proof}

\begin{lemma}\label{lemma 4}
For $\mathcal{A}=\mathcal{A}_1\otimes\mathcal{A}_2$,
\begin{enumerate}
\item $\left(\Omega_{D_1}^1(\mathcal{A}_1)\otimes\mathcal{A}_2\right)\otimes_\mathcal{A}(\mathcal{E}_1\otimes\mathcal{E}_2)\cong
\left(\Omega_{D_1}^1(\mathcal{A}_1)\otimes_{\mathcal{A}_1}\mathcal{E}_1\right)\otimes_\mathbb{C}\mathcal{E}_2$ as $\mathcal{A}$-bimodule.
\item $(\mathcal{E}_1\otimes\mathcal{E}_2)\otimes_\mathcal{A}\left(\Omega_{D_1}^1(\mathcal{A}_1)\otimes\mathcal{A}_2\right)\cong
\left(\mathcal{E}_1\otimes_{\mathcal{A}_1}\Omega_{D_1}^1(\mathcal{A}_1)\right)\otimes_\mathbb{C}\mathcal{E}_2$ as $\mathcal{A}$-bimodule.
\item $\left(\mathcal{A}_1\otimes\Omega_{D_2}^1(\mathcal{A}_2)\right)\otimes_\mathcal{A}(\mathcal{E}_1\otimes\mathcal{E}_2)\cong
\mathcal{E}_1\otimes_\mathbb{C}\left(\Omega_{D_2}^1(\mathcal{A}_2)\otimes_{\mathcal{A}_2}\mathcal{E}_2\right)$ as $\mathcal{A}$-bimodule.
\item $(\mathcal{E}_1\otimes\mathcal{E}_2)\otimes_\mathcal{A}\left(\mathcal{A}_1\otimes\Omega_{D_2}^1(\mathcal{A}_2)\right)\cong
\mathcal{E}_1\otimes_\mathbb{C}\left(\mathcal{E}_2\otimes_{\mathcal{A}_2}\Omega_{D_2}^1(\mathcal{A}_2)\right)$ as $\mathcal{A}$-bimodule.
\end{enumerate}
\end{lemma}
\begin{proof}
These are canonical isomorphisms since both $\mathcal{A}_1,\mathcal{A}_2$ are unital algebras, and $\mathcal{E}_j\otimes_{\mathcal{A}_j}\mathcal{A}_j
\cong\mathcal{E}_j$ for $j=1,2$.
\end{proof}

\textbf{Notation:} Throughout this section $\,\mathcal{A}:=\mathcal{A}_1\otimes\mathcal{A}_2\,,\,\mathcal{E}:=\mathcal{E}_1\otimes\mathcal{E}_2\,,\,D=
D_1\otimes 1+\gamma_1\otimes D_2\,,\,\gamma=\gamma_1\otimes\gamma_2$ and $J=J_1\otimes J_2$.
\medskip

We will use Lemmas (\ref{first form}\,,\,\ref{first differential}\,,\,\ref{lemma 3}\,,\,\ref{lemma 4}) frequently in several places in this section
without any further mention. Now, for two compatible connections $\nabla_1\in C(\mathcal{E}_1)$ and $\nabla_2\in C(\mathcal{E}_2)$ define
\begin{align*}
\nabla:\mathcal{E}_1\otimes\mathcal{E}_2 &\longrightarrow \Omega_D^1(\mathcal{A})\otimes_\mathcal{A}(\mathcal{E}_1\otimes\mathcal{E}_2)\\
e_1\otimes e_2 &\longmapsto \nabla_1(e_1)\otimes e_2+e_1\otimes\nabla_2(e_2)
\end{align*}

\begin{proposition}\label{left product connection}
$\nabla\in C(\mathcal{E}_1\otimes\mathcal{E}_2)$, i,e. if $\nabla_1,\nabla_2$ are compatible connections on $\mathcal{E}_1$ and $\mathcal{E}_2$
respectively then so is $\nabla$ on $\mathcal{E}=\mathcal{E}_1\otimes\mathcal{E}_2$. Moreover, $\nabla(\gamma(e_1\otimes e_2))=(1\otimes\gamma)
\nabla(e_1\otimes e_2)$, i,e. $\nabla$ commutes with $\gamma$ on $\mathcal{E}_1\otimes\mathcal{E}_2$.
\end{proposition}
\begin{proof}
Clearly $\nabla$ is a $\mathbb{C}$-linear map. Now, for $e_1\otimes e_2\in\mathcal{E}$ and $x\otimes y\in\mathcal{A}$
\begin{eqnarray*}
\nabla((x\otimes y)(e_1\otimes e_2)) & = & \nabla_1(xe_1)\otimes ye_2+xe_1\otimes\nabla_2(ye_2)\\
& = & x\nabla_1(e_1)\otimes ye_2+(d_1x\otimes e_1)\otimes ye_2+xe_1\otimes y\nabla_2(e_2)+xe_1\otimes(d_2y\otimes e_2)\\
& = & (x\nabla_1(e_1)\otimes ye_2+xe_1\otimes y\nabla_2(e_2))+(d_1x\otimes 1\otimes e_1\otimes ye_2+1\otimes d_2y\otimes xe_1\otimes e_2)\\
& = & (x\otimes y)(\nabla_1(e_1)\otimes e_2+e_1\otimes\nabla_2(e_2))+(d_1x\otimes y+x\otimes d_2y)\otimes(e_1\otimes e_2)\\
& = & (x\otimes y)\nabla(e_1\otimes e_2)+d(x\otimes y)\otimes(e_1\otimes e_2)
\end{eqnarray*}
by Lemma (\ref{first differential}). Hence, $\nabla$ is a connection on $\mathcal{E}=\mathcal{E}_1\otimes\mathcal{E}_2$. Now to check the
compatibility of $\nabla$ with respect to the Hermitian structure on $\mathcal{E}$, write
\begin{align*}
\nabla_j(e_j) &= \sum_i\,\omega_{ji}\otimes e_{ji}\,\,\in\Omega_{D_j}^1(\mathcal{A}_j)\otimes\mathcal{E}_j\\
\nabla_j(e_j^\prime) &= \sum_i\,\omega_{ji}^\prime\otimes e_{ji}^\prime\in\Omega_{D_j}^1(\mathcal{A}_j)\otimes\mathcal{E}_j\\
\end{align*}
for $j=1,2$. Then
\begin{align*}
\nabla_1(e_1)\otimes e_2+e_1\otimes\nabla_2(e_2) &= \sum_i\,(\omega_{1i}\otimes e_{1i}\otimes e_2\,,\,\omega_{2i}\otimes e_1\otimes e_{2i})\,,\\
\nabla_1(e_1^\prime)\otimes e_2^\prime+e_1^\prime\otimes\nabla_2(e_2^\prime) &= \sum_i\,(\omega_{1i}^\prime\otimes e_{1i}^\prime\otimes e_2^\prime
\,,\,\omega_{2i}^\prime\otimes e_1^\prime\otimes e_{2i}^\prime)\,,
\end{align*}
and
\begin{eqnarray}\label{to be used in last line}
d\langle e_1\otimes e_2\,,\,e_1^\prime\otimes e_2^\prime\rangle & = & d(\langle e_1,e_1^\prime\rangle\otimes\langle e_2,e_2^\prime\rangle)\\
& = & d_1(\langle e_1,e_1^\prime\rangle)\otimes\langle e_2,e_2^\prime\rangle+\langle e_1,e_1^\prime\rangle\otimes d_2(\langle e_2,e_2^\prime\rangle)\,.\nonumber
\end{eqnarray}
Since, $\nabla_1\in C(\mathcal{E}_1)$ and $\nabla_2\in C(\mathcal{E}_2)$ we have
\begin{center}
$-\langle e_j,\nabla_je_j^\prime\rangle+\langle\nabla_je_j,e_j^\prime\rangle=d_j(\langle e_j,e_j^\prime\rangle)$
\end{center}
for $j=1,2$, which further implies the following
\begin{eqnarray}\label{individual compatibility}
\sum_i-\langle e_j,e_{ji}^\prime\rangle(\omega_{ji}^\prime)^*+\omega_{ji}\langle e_{ji},e_j^\prime\rangle=d_j(\langle e_j,e_j^\prime\rangle)\,.
\end{eqnarray}
for $j=1,2$. Now,
\begin{eqnarray*}
&  & \langle e_1\otimes e_2\,,\,\nabla(e_1^\prime\otimes e_2^\prime)\rangle\\
& = & \langle e_1\otimes e_2\,,\,\sum_i\omega_{1i}^\prime\otimes e_{1i}^\prime\otimes e_2^\prime+\omega_{2i}^\prime\otimes e_1^\prime\otimes e_{2i}^\prime\rangle\\
& = & \sum_{i,j}\langle e_1\otimes e_2\,,\,a_{01ij}^\prime[D_1,a_{11ij}^\prime]\otimes e_{1i}^\prime\otimes e_2^\prime+a_{02ij}^\prime[D_2,a_{12ij}^\prime]
\otimes e_1^\prime\otimes e_{2i}^\prime\rangle\\
& = & \sum_{i,j}\langle e_1\otimes e_2\,,\,((a_{01ij}^\prime\otimes 1)[D,a_{11ij}^\prime\otimes 1])\otimes e_{1i}^\prime\otimes e_2^\prime+
((1\otimes a_{02ij}^\prime)[D,1\otimes a_{12ij}^\prime])\otimes e_1^\prime\otimes e_{2i}^\prime\rangle\\
& = & \sum_{i,j}(\langle e_1,e_{1i}^\prime\rangle\otimes\langle e_2,e_2^\prime\rangle)((a_{01ij}^\prime\otimes 1)[D,a_{11ij}^\prime\otimes 1])^*
+\sum_{i,j}(\langle e_1,e_1^\prime\rangle\otimes\langle e_2,e_{2i}^\prime\rangle)((1\otimes a_{02ij}^\prime)[D,1\otimes a_{12ij}^\prime])^*\\
& = & \sum_{i,j}\langle e_1,e_{1i}^\prime\rangle[D_1,a_{11ij}^\prime]^*(a_{01ij}^\prime)^*\otimes\langle e_2,e_2^\prime\rangle+\langle e_1,e_1^\prime
\rangle\otimes\langle e_2,e_{2i}^\prime\rangle[D_2,a_{12ij}^\prime]^*(a_{02ij}^\prime)^*\\
& = & \sum_i\langle e_1,e_{1i}^\prime\rangle(\omega_{1i}^\prime)^*\otimes\langle e_2,e_2^\prime\rangle+\langle e_1,e_1^\prime\rangle\otimes
\langle e_2,e_{2i}^\prime\rangle(\omega_{2i}^\prime)^*\,.
\end{eqnarray*}
Similarly, one can show that
\begin{center}
$\langle\nabla(e_1\otimes e_2)\,,\,e_1^\prime\otimes e_2^\prime\rangle=\sum_i\omega_{1i}\langle e_{1i},e_1^\prime\rangle\otimes\langle e_2,
e_2^\prime\rangle+\langle e_1,e_1^\prime\rangle\otimes\omega_{2i}\langle e_{2i},e_2^\prime\rangle\,.$
\end{center}
Subtracting we get
\begin{center}
$\langle\nabla(e_1\otimes e_2)\,,\,e_1^\prime\otimes e_2^\prime\rangle-\langle e_1\otimes e_2\,,\,\nabla(e_1^\prime\otimes e_2^\prime)\rangle=
d(\langle e_1\otimes e_2\,,\,e_1^\prime\otimes e_2^\prime\rangle)$
\end{center}
by equation (\ref{to be used in last line}) and (\ref{individual compatibility}). This proves that $\nabla$ is a compatible connection, i,e.
$\nabla\in C(\mathcal{E})$. That $\nabla$ commutes with $\gamma$ on $\mathcal{E}_1\otimes\mathcal{E}_2$ is easy to verify.
\end{proof}

\begin{lemma}\label{flip operator}
The flip operator $\Psi:\Omega_D^1(\mathcal{A})\otimes_\mathcal{A}\mathcal{E}\longrightarrow\mathcal{E}\otimes_\mathcal{A}\Omega_D^1(\mathcal{A})$
is given by $\Psi_1\otimes J_2\bigoplus J_1\otimes\Psi_2$.
\end{lemma}
\begin{proof}
Let $\xi=\xi_1\otimes\xi_2\in\mathcal{E}=\mathcal{E}_1\otimes\mathcal{E}_2$ and $\omega=(\omega_1\otimes a_2,a_1\otimes\omega_2)\in\Omega_D^1
(\mathcal{A})$. Then,
\begin{center}
$\omega\otimes\xi\,=\,(\omega_1\otimes\xi_1\otimes a_2\xi_2\,,\,a_1\xi_1\otimes\omega_2\otimes\xi_2)\,.$
\end{center}
Hence,
\begin{eqnarray*}
J\xi\otimes\omega^* & = & J(\xi_1\otimes\xi_2)\otimes(\omega_1^*\otimes a_2^*\,,\,a_1^*\otimes\omega_2^*)\\
& = & (J_1\xi_1\otimes J_2\xi_2)\otimes(\omega_1^*\otimes a_2^*\,,\,a_1^*\otimes\omega_2^*)\\
& = & (J_1\xi_1\otimes\omega_1^*\otimes(J_2\xi_2)a_2^*\,,\,(J_1\xi_1)a_1^*\otimes J_2\xi_2\otimes\omega_2^*)
\end{eqnarray*}
Now, observe that
\begin{eqnarray*}
J(\xi)\,.\,a^* & = & JaJ^*J(\xi)\\
& = & Ja(\xi)\\
& = & J(a\xi)\,.
\end{eqnarray*}
Since, by definition
\begin{align*}
\Psi:\Omega_D^1(\mathcal{A})\otimes_\mathcal{A}\mathcal{E} &\longrightarrow \mathcal{E}\otimes_\mathcal{A}\Omega_D^1(\mathcal{A})\\
\omega\otimes\xi &\longmapsto J\xi\otimes\omega^*
\end{align*}
we see that $\Psi=\Psi_1\otimes J_2\bigoplus J_1\otimes\Psi_2\,$.
\end{proof}

\begin{lemma}\label{right product connection}
The right connection $\overline{\nabla}:\mathcal{E}\longrightarrow\mathcal{E}\otimes_\mathcal{A}\Omega_D^1(\mathcal{A})$ is given by
$\overline{\nabla}_1\otimes 1+1\otimes\overline{\nabla}_2$.
\end{lemma}
\begin{proof}
Follows from previous Lemma (\ref{flip operator}) and the fact that $J=J_1\otimes J_2$.
\end{proof}

\begin{lemma}\label{tensored connection}
The ``tensored connection'' $\widetilde{\nabla}:\mathcal{E}\otimes_\mathcal{A}\mathcal{E}\longrightarrow\mathcal{E}\otimes_\mathcal{A}\Omega_D^1
(\mathcal{A})\otimes_\mathcal{A}\mathcal{E}$ is given by
\begin{center}
$(\xi_1\otimes\xi_2\otimes\eta_1\otimes\eta_2)\longmapsto\left(\overline{\nabla}_1\otimes\tau\otimes 1+1\otimes(\nabla_1\otimes 1)\tau\otimes 1\,,\,1
\otimes(1\otimes\overline{\nabla}_2)\tau\otimes 1+1\otimes\tau\otimes\nabla_2\right)$
\end{center}
where, $\tau:\xi_2\otimes\eta_1\longmapsto\eta_1\otimes\xi_2$ is the usual flip of tensor product.
\end{lemma}
\begin{proof}
For $\xi=\xi_1\otimes\xi_2$ and $\eta=\eta_1\otimes\eta_2$ in $\,\mathcal{E}$, using Lemma (\ref{right product connection}) and Proposition
(\ref{left product connection}) we get
\begin{eqnarray*}
&  & \widetilde{\nabla}(\xi\otimes\eta)\\
& = & \overline{\nabla}\xi\otimes\eta+\xi\otimes\nabla\eta\\
& = & \left(\,\overline{\nabla}_1\xi_1\otimes\xi_2\,,\,\xi_1\otimes\overline{\nabla}_2\xi_2\right)\otimes\eta+\xi\otimes\left(\nabla_1\eta_1
\otimes\eta_2\,,\,\eta_1\otimes\nabla_2\eta_2\right)\\
& = & \left(\,\overline{\nabla}_1\xi_1\otimes\xi_2\otimes\eta_1\otimes\eta_2+\xi_1\otimes\xi_2\otimes\nabla_1\eta_1\otimes\eta_2\,,\,\xi_1\otimes
\overline{\nabla}_2\xi_2\otimes\eta_1\otimes\eta_2+\xi_1\otimes\xi_2\otimes\eta_1\otimes\nabla_2\eta_2\right)\\
& = & ((\,\overline{\nabla}_1\xi_1\otimes\eta_1)\otimes(\xi_2\otimes\eta_2)+(\xi_1\otimes\nabla_1\eta_1)\otimes(\xi_2\otimes\eta_2)\,,\,(\xi_1
\otimes\eta_1)\otimes(\,\overline{\nabla}_2\xi_2\otimes\eta_2)\\
&  & \hspace*{8cm}+(\xi_1\otimes\eta_1)\otimes(\xi_2\otimes\nabla_2\eta_2))\\
\end{eqnarray*}
and this concludes the proof.
\end{proof}

\begin{lemma}\label{linear maps c}
The $\mathbb{C}$-linear maps
\begin{center}
$c^\prime\,,\,\overline{c}^{\,\prime}\,:\,\mathcal{E}\otimes_\mathcal{A}\Omega_D^1(\mathcal{A})\otimes_\mathcal{A}\mathcal{E}\longrightarrow\mathcal{E}
\otimes_\mathcal{A}\mathcal{E}$
\end{center}
are given by
\begin{center}
$c^\prime=c_1\otimes 1+\star_1\otimes c_2\quad$ and $\quad\overline{c}^{\,\prime}=\overline{c}_1\otimes\star_2+(\gamma_1\otimes\gamma_1)\otimes
\overline{c}_2\,.$
\end{center}
\end{lemma}
\begin{proof}
Consider $\xi_1=\xi_{11}\otimes\xi_{12}$ and $\xi_2=\xi_{21}\otimes\xi_{22}$ in $\mathcal{E}$. Let $\omega=(\omega_1\otimes a_2\,,\,a_1\otimes
\omega_2)\in\Omega_D^1(\mathcal{A})$. Now,
\begin{eqnarray*}
\xi_1\otimes\omega\otimes\xi_2 & = & (\xi_{11}\otimes\xi_{12})\otimes(\omega_1\otimes a_2\,,\,a_1\otimes\omega_2)\otimes(\xi_{21}\otimes\xi_{22})\\
& = & (\xi_{11}\otimes\omega_1\otimes\xi_{12}a_2\,,\,\xi_{11}a_1\otimes\xi_{12}\otimes\omega_2)\otimes(\xi_{21}\otimes\xi_{22})\\
& = & (\xi_{11}\otimes\omega_1\otimes\xi_{21}\otimes\xi_{12}a_2\otimes\xi_{22}\,,\,\xi_{11}a_1\otimes\xi_{21}\otimes\xi_{12}\otimes\omega_2\otimes\xi_{22})
\end{eqnarray*}
and
\begin{eqnarray*}
\xi_1\otimes\omega\,.\,\xi_2 & = & (\xi_{11}\otimes\xi_{12})\otimes(\omega_1\otimes a_2\,,\,a_1\otimes\omega_2)\,.\,(\xi_{21}\otimes\xi_{22})\\
& = & (\xi_{11}\otimes\xi_{12})\otimes(\omega_1.\,\xi_{21}\otimes a_2\xi_{22}+\gamma_1a_1\xi_{21}\otimes\omega_2\,.\,\xi_{22})\\
& = & \xi_{11}\otimes\omega_1.\,\xi_{21}\otimes\xi_{12}\otimes a_2\xi_{22}+\xi_{11}\otimes\gamma_1a_1\xi_{21}\otimes\xi_{12}\otimes\omega_2\,.\,\xi_{22}
\end{eqnarray*}
Since, $c^\prime:\xi_1\otimes\omega\otimes\xi_2\longmapsto\xi_1\otimes\omega\,.\,\xi_2$ and $\star_1=1\otimes\gamma_1$ we get
\begin{center}
$c^\prime=c_1\otimes 1+\star_1\otimes c_2\,.$
\end{center}
Similarly, one can verify that $\overline{c}^{\,\prime}:\xi_1\otimes\omega\otimes\xi_2\longmapsto\xi_1.\,\omega\otimes\gamma\xi_2$ is given by
\begin{center}
$\overline{c}^\prime=\overline{c}_1\otimes\star_2+(\gamma_1\otimes\gamma_1)\otimes\overline{c}_2$
\end{center}
where $\star_2=1\otimes\gamma_2$ and $\gamma_1\otimes\gamma_1$ is the grading operator on $\overline{\mathcal{E}_1\otimes_{\mathcal{A}_1}\mathcal{E}_1}\,$.
\end{proof}

\begin{lemma}\label{final Dirac operators for product}
We have
\begin{itemize}
\item[(i)] $\mathfrak{D}:=c^\prime\circ\widetilde{\nabla}=\mathfrak{D}_1\otimes 1+\star_1\otimes\mathfrak{D}_2$
\item[(ii)] $\overline{\mathfrak{D}}:=\overline{c}^\prime\circ\widetilde{\nabla}=\overline{\mathfrak{D}}_1\otimes\star_2+\widetilde{\gamma_1}\otimes
\overline{\mathfrak{D}}_2$
\end{itemize}
where, $\widetilde{\gamma_1}=\gamma_1\otimes\gamma_1$ is the grading operator acting on $\overline{\mathcal{E}_1\otimes_{\mathcal{A}_1}\mathcal{E}_1}\,$.
\end{lemma}
\begin{proof}
Consider $\xi=\xi_1\otimes\xi_2$ and $\eta=\eta_1\otimes\eta_2$ in $\mathcal{E}$. Now, using Lemma (\ref{tensored connection}\,,\,\ref{linear maps c})
we get
\begin{eqnarray*}
\mathfrak{D} & = & c^\prime\circ\widetilde{\nabla}(\xi\otimes\eta)\\
& = & \left(c_1(\,\overline{\nabla}_1\xi_1\otimes\eta_1)+c_1(\xi_1\otimes\nabla_1\eta_1)\right)\otimes(\xi_2\otimes\eta_2)\\
&  & +\star_1(\xi_1\otimes\eta_1)\otimes\left(c_2\,(\overline{\nabla}_2\xi_2\otimes\eta_2)+c_2(\xi_2\otimes\nabla_2\eta_2)\right)\\
& = & c_1\widetilde{\nabla_1}(\xi_1\otimes\eta_1)\otimes(\xi_2\otimes\eta_2)+\star_1(\xi_1\otimes\eta_1)\otimes c_2\widetilde{\nabla_2}(\xi_2\otimes\eta_2)\\
& = & (\mathfrak{D}_1\otimes 1+\star_1\otimes\mathfrak{D}_2)(\xi\otimes\eta)
\end{eqnarray*}
Similarly, one can show that $\overline{\mathfrak{D}}=\overline{\mathfrak{D}}_1\otimes\star_2+(\gamma_1\otimes\gamma_1)\otimes\overline{\mathfrak{D}}_2\,$.
\end{proof}

\begin{lemma}\label{relations among Dirac operators}
Both $\mathfrak{D}$ and $\overline{\mathfrak{D}}$ are essentially self-adjoint operator satisfying $\mathfrak{D}^2=\overline{\mathfrak{D}}^2$ and
$\{\mathfrak{D},\overline{\mathfrak{D}}\}=0$.
\end{lemma}
\begin{proof}
Essential self-adjointness follows from the expression of $\mathfrak{D}$ and $\overline{\mathfrak{D}}$ in Lemma (\ref{final Dirac operators for product})
along the line of (\cite{DD}, Page $1839$). The relations $\mathfrak{D}^2=\overline{\mathfrak{D}}^2$ and $\{\mathfrak{D},\overline{\mathfrak{D}}\}=0$
follow from Lemma (\ref{crucial relation involving Hodge}) and the fact that $[\star_1,\widetilde{\gamma_1}]=0$.
\end{proof}

\begin{lemma}\label{trace class}
We have
\begin{itemize}
\item[(i)] $[\mathfrak{D},a]$ and $[\,\overline{\mathfrak{D}},a]$ extends to bounded operators on $\overline{\mathcal{E}\otimes_\mathcal{A}
\mathcal{E}}$ for all $a\in\mathcal{A};$
\item[(ii)] $exp(-\varepsilon\mathfrak{D}^2)$ is trace class for all $\,\varepsilon>0$.
\end{itemize}
\end{lemma}
\begin{proof}
Since $\{\star_1,\mathfrak{D}_1\}=0,\,\mathfrak{D}^2=\mathfrak{D}_1^2\otimes 1+1\otimes\mathfrak{D}_2^2\,$. This shows that $exp(-\varepsilon
\mathfrak{D}^2)$ is trace class for all $\varepsilon>0$ since, for $j=1,2,\,exp(-\varepsilon\mathfrak{D}_j^2)$ is trace class for all $\varepsilon>0$
by our assumption. Checking the bounded commutators are easy.
\end{proof}

Combining Lemma (\ref{final Dirac operators for product}\,,\,\ref{relations among Dirac operators}\,,\,\ref{trace class}) we conclude the following
theorem.

\begin{theorem}
Given two $N=1$ spectral data $(\mathcal{A}_j,\mathcal{H}_j,D_j,\gamma_j),\,j=1,2$, if $\,\Phi(\mathcal{A}_j,\mathcal{H}_j,D_j,\gamma_j)$ gives us
two $N=(1,1)$ spectral data then $\Phi\left(\otimes_{j=1}^2(\mathcal{A}_j,\mathcal{H}_j,D_j,\gamma_j)\right)$ is also a $N=(1,1)$ spectral data.
Moreover, $\Phi$ is multiplicative, i,e.
\begin{center}
$\Phi\left(\otimes_{j=1}^2(\mathcal{A}_j,\mathcal{H}_j,D_j,\gamma_j)\right)=\otimes_{j=1}^2\Phi\left((\mathcal{A}_j,\mathcal{H}_j,D_j,\gamma_j)
\right)\,,$
\end{center}
w.r.t the tensor product in Def. $(\ref{product of N=(1,1) data})$ but not multiplicative w.r.t any other tensor product in Proposition
$(\ref{various product of N=(1,1)})$. Therefore, if we demand that the extension procedure $\Phi$ is multiplicative then there is a unique
choice of tensor product of $N=(1,1)$ spectral data.
\end{theorem}
\medskip

\subsection*{Acknowledgement}
Author gratefully acknowledges financial support of DST, India through\\
INSPIRE Faculty award (Award No. DST/INSPIRE/04/2015/000901).

\end{document}